\documentclass{amsart}
\usepackage{amssymb,amscd,color}
\usepackage{tikz, tikz-cd}
\usepackage{enumitem}
\usepackage{acc}
\usepackage[normalem]{ulem}
\usepackage{color}

\usepackage[
urlcolor=blue, linktocpage, hyperindex=true, colorlinks=true,
linkcolor=blue, citecolor=blue,
]{hyperref}
\newcommand\acc{\operatorname{Acc}}
\newcommand\omc{\overline{\mc}}
\newcommand\Sing{\operatorname{Sing}}
\newcommand\Diff{\operatorname{Diff}}
\newcommand{\klt}{\mathrm{klt}}

\newcommand\FF{\mathbb{F}}

\newcommand\kk{\bK^2}

\begin{document}
\title{On accumulation points of volumes of log surfaces}
\author{Valery Alexeev}
\address{Valery Alexeev\\Department of Mathematics\\ University of Georgia\\ Athens\\GA 30605\\ USA}
\email{valery@math.uga.edu}
\author{Wenfei Liu}
\address{Wenfei Liu \\School of Mathematical 
Sciences\\ Xiamen University\\Siming South Road 422\\Xiamen\\ Fujian 361005\\ P.~R.~China}
\email{wliu@xmu.edu.cn}
\subjclass[2010]{Primary 14J29; Secondary 14J26, 14R05}
\date{March 25, 2018}
\keywords{log canonical surfaces, volume, accumulation points}

\begin{abstract}
  Let $\mc\subset(0,1]$ be a set satisfying the descending chain
  condition. We show that any accumulation point of volumes of log
  canonical surfaces $(X, B)$ with coefficients in $\mc$ can be
  realized as the volume of a log canonical surface with big and nef
  $K_X+B$ and coefficients in $\overline\mc\cup\{1\}$, with at least
  one coefficient in $\acc(\mc)\cup\{1\}$.  As a corollary, if
  $\overline{\mc}\subset\bQ$ then all accumulation points of volumes
  are rational numbers, solving a conjecture of Blache.  For the
  set of standard coefficients
  $\cC_2=\{1-\frac{1}{n}\mid n\in\N\}\cup\{1\}$ we prove
  that the minimal accumulation point is between
  $\frac1{7^2\cdot 42^2}$ and $\frac1{42^2}$.
\end{abstract}
\maketitle
\tableofcontents

\section{Introduction}
Let $(X, B)$ be a projective log surface with log canonical 
singularities. The volume $\vol(K_X+B)$ measures the
asymptotic growth of the space of global sections of pluri-log-canonical
divisors:
\[
h^0(X, \lfloor m(K_X+B)\rfloor) = \frac{\vol(K_X+B)}{2} m^2 + \mathrm{o}(m^2)
\]
One says
that $K_X+B$ is big if $\vol(K_X+B)>0$.  If this is the case, let
$\pi\colon (X, B)\rightarrow (X_\can, B_\can)$ be the contraction to
the log canonical model. Then $\pi^*(K_{X_\can}+B_\can)$ is the
positive part of $K_X+B$ in the Zariski decomposition and
 $$\vol(K_X+ B)=\vol(K_{X_\can}+B_\can)=(K_{X_\can}+B_\can)^2.$$ The log canonical divisor $K_{X_\can}+B_\can$ is ample, but not Cartier in general. Hence $\vol(K_X+B)$, being positive, is not necessarily an integer. 
 
 Assume that $\mc\subset (0,1]$ and let $\ms(\mc)$ be the set of log
 canonical projective surfaces $(X, B)$ such that the coefficients of
 $B$ belong to $\mc$. Using the minimal model program for log
 surfaces, one has the following equalities of sets of volumes
\begin{align*}
\kk(\mc) :&=\{(\vol(K_X+B) \mid (X, B) \in\ms(\mc), K_X+B \text{ big} \}\\
&=\{(K_X+B)^2 \mid (X, B) \in\ms(\mc), K_X+B  \text{ big and nef} \}\\
& = \{(K_X+B)^2 \mid (X, B) \in\ms(\mc), K_X+B \text{ ample} \}
\end{align*}

The three most commonly used sets of coefficients are
$\mc_0=\emptyset$, $\mc_1 = \{1\}$, and
$\mc_2=\{1-\frac{1}{n}\mid n\in\N\}\cup\{1\}$. The latter set is
called the set of \emph{standard coefficients}. It naturally appears when one
studies groups of automorphisms of smooth varieties and in the
adjunction formula. An easy observation is that these three sets
satisfy the descending chain condition (DCC).
 
It is a fundamental result of the first author that if $\mc$ is a DCC
set then the set $\kk(\mc)$ is also DCC
(\cite[Thm.~8.2]{alexeev1994boundedness-and-ksp-2}).  The DCC is an
assertion about the accumulation points of $\kk(\mc)$.  Blache
\cite{blache1995example-concerning} and Koll\'ar
\cite[Theorem~39]{kollar2008is-there} have constructed several
examples of accumulation points of volumes of log surfaces without
boundary, and recently Urz\'ua and Y\'a\~nez
\cite{urzua2017notes-accumulation} exploited Koll\'ar's construction
to show that all natural numbers are actually accumulation points of
$\kk(\mc_0)$. Yet the overall picture of the accumulation points is
unknown even for $\mc_0$.

We recall the definition of an nklt center of $(X,B)$ in
Section~\ref{sec:preliminaries}. We call an nklt center
\emph{accessible}
if it is not a simple elliptic singularity $P\in X$ or
a smooth curve $B_0\subset B$ disjoint from $\Sing(X)$ and the rest of
$B$.  The first main result of this paper is the following.

\begin{thm}\label{thm: main}
  Suppose that $\mc\subset(0,1]$ satisfies the descending chain
  condition.  Then $v_\infty\in \RR_{>0}$ is an accumulation point of
  $\kk(\mc)$ if and only if there exists a log canonical surface 
  $(X, B)\in \ms(\overline\mc\cup\{1\})$ such that
\begin{enumerate} 
\item $K_X+B$ is ample, and $v_\infty=(K_X+B)^2$.
\item One of the following conditions is satisfied:
\begin{enumerate}
\item The set of codiscrepancies of divisors \emph{over} $X$ with
  respect to $(X, B)$ contains an accumulation point of $\mc$.
\item $(X,B)$ has an accessible nklt center. 
\end{enumerate}
\item If $1$ is not in the closure $\overline\cC$ of $\cC$  then each irreducible component of $\lfloor B\rfloor$ has geometric
  genus at most 1. 
\end{enumerate}
\end{thm}

\begin{rmk}
  An equivalent condition for (ii.a) is that there exists a log
  surface $(X,B)\in \cS(\overline\cC\cup 1)$ with big and nef $K_X+B$ and a
  coefficient of $B$ is in $\acc(\cC)$. 
\end{rmk}

In Section~\ref{sec: constr} we show the "if part" of
Theorem~\ref{thm: main} by constructing a sequence of log canonical
surfaces with volumes converging to $v_\infty$. The harder part of the
theorem is the "only if part". The proof is carried out in
Section~\ref{sec: limit surf}. 
There are several interesting consequences of Theorem~\ref{thm: main}.
\begin{cor}\label{cor: rat}
Let $\mc\subset(0,1]$ be a DCC set. Then $\acc(\kk(\mc))\subset
\kk(\overline \mc\cup\{1\})$, where $\acc(\cdot)$ denotes the set of
accumulation points. In particular, if $\omc\subset\bQ$ then
$\acc(\kk(\mc))\subset \bQ$.
\end{cor}

\begin{rmk}
  Theorem~\ref{thm: main} gives an easy way to find many accumulation
  points of $\kk(\mc)$. E.g., it easily follows that
  $ \Z_{>0}\subset\acc(\kk(\mc_0))$, cf. Example~\ref{ex: Z}.

  Corollary~\ref{cor: rat} and Example~\ref{ex: Z} prove one half of a
  conjecture of Blache \cite[Conjecture
  3(a)]{blache1995example-concerning}.  On the other hand, the second
  half \cite[Conjecture 3(b)]{blache1995example-concerning} saying
  that $\min \acc(\kk(\mc_0))= 1$ is false. For example, the surface
  in \cite[Sec.5]{alexeev2016open-surfaces} of volume $\frac{1}{462}$
  has a non-empty accessible nklt locus as in Theorem~\ref{thm:
    main}(ii.b), and one infers that $\frac1{462}$ is an accumulation
  point of $\kk(\mc_0)$.

  The fact that $\Z_{>0}\subset\acc(\kk(\mc_0))$ and that there are
  many accumulation points smaller than one was also established in
  the recent work of Urz\'ua and Y\'a\~nez
  \cite{urzua2017notes-accumulation}.
\end{rmk}

\begin{cor}\label{cor: =}
  Let $\mc\subset(0,1]$ be a DCC set such that $\acc(\cC) = \{1\}$,
  e.g. $\cC=\cC_2$.  Then $\acc({\kk(\mc)})=\kk_\nklt(\overline \mc)$,
  where
  $$\kk_\nklt(\overline\mc) = \{(K_X+B)^2 \mid (X, B)
  \in\ms(\overline\mc),\  K_X+B \text{ ample},\  \nklt(X, B)\neq \emptyset
  \}.$$
  If $\cC$ is finite, e.g. $\cC=\cC_0$ or $\cC_1$,
  then $\acc({\kk(\mc)}) \subset \kk_\nklt(\mc)$.
\end{cor}

Our second main result is an explicit lower bound for
$\acc(\kk(\mc_2))$ for the standard coefficient set
$\cC_2 = \{1-\frac1{n}, \ n\in\bN\} \cup \{1\}$. In Section~\ref{sec:
  low bnd} we prove:
\begin{thm}\label{thm:bounds12}
  One has 
  $\min \acc(\cC_2) =
  \min \kk_\nklt(\mc_2)\geq \frac{1}{86436} = \frac{1}{7^2\cdot42^2}$. 
\end{thm}

On the other hand, $\frac{1}{42^2}$ is an accumulation point of
$\kk(\mc_2)$ by applying Theorem~\ref{thm: main}(ii.b) to
\cite[Example~5.3.1]{kollar1994log-surfaces}. Overall, our bounds for
the minimal accumulation points of the three most commonly used sets
are as follows:

\begin{thm}\label{thm:bounds}
  The following is true:
  \begin{enumerate}

    \smallskip\item For $\cC=\cC_0$, $\cC_1$ and $\cC_2$, ones has
    $\min \acc({\kk(\mc)})= \min\kk_\nklt(\mc)$.

  \medskip\item One has
    \(
    \frac{1}{86436} = \frac{1}{7^2\cdot42^2} \le
    \min \kk_\nklt(\mc_2) \le
    \frac{1}{1764} = \frac{1}{42^2}.
    \)
    
  \medskip\item For $\cC=\cC_0$ or $\cC_1$, one has
    \(
    \frac{1}{86436} = \frac{1}{7^2\cdot42^2} \le
    \min \kk_\nklt(\mc) \le
    \frac{1}{462} = \frac{1}{11\cdot 42}.
    \)    
  \end{enumerate}

\end{thm}

\begin{rmk}
  \cite[4.8]{alexeev2004bounding-singular} gives an effective lower
  bound for $\kk(\mc)$ for any DCC set $\mc$, which is however too
  small to be realistic.  For the sets $\mc_0,\mc_1,\mc_2$ it works
  out to about $10^{-3\cdot 10^{10}}$,
  cf.~\cite[Sec.10]{alexeev2016open-surfaces}.  Our bounds for
  $\acc(\kk(\mc_2))$ are considerably "larger".
\end{rmk}

In Section~\ref{sec:iterated} we discuss the iterated accumulation
points of $\bK^2(\cC)$, i.e. accumulation points of accumulation
points etc. In particular, we establish that its ``accumulation
complexity'' is unbounded already for the simplest set
$\cC_0=\emptyset$.

\begin{ack}
  The work of the first author was partially supported by NSF under
  DMS-1603604. He would like to thank Christopher Hacon and Chenyang
  Xu for useful discussions.
The second author was partially supported by the NSFC
  (No.~11501012, No.~11771294) and by the Recruitment Program for
  Young Professionals. He benefited from discussions with Professor
  Meng Chen during a visit to Fudan University in September
  2017. Thanks go to Giancarlo Urz\'ua for sending us the manuscript
  \cite{urzua2017notes-accumulation}.
\end{ack}

\section{Preliminaries}
\label{sec:preliminaries} 

Throughout the paper, we work over an algebraically closed field of
arbitrary characteristic.
We recall the standard definitions of the MMP, see
e.g. \cite{kollar1998birational-geometry}. 
We will only need them for surfaces.
Given an $\R$-divisor $B=\sum_j b_j B_j$ on a normal surface $X$ and  a real number $a$ we write
\begin{displaymath}
  B^{= a} = \sum_{b_j= a} b_jB_j, \quad
  B^{>a} = \sum_{b_j>a} b_jB_j,\quad
  B^{<a} = \sum_{b_j<a} b_jB_j.   
\end{displaymath}

A log surface consists of a normal projective surface $X$ and
effective $\R$-divisor~$B$
such that $K_X+B$ is $\R$-Cartier.  Let
$f\colon Y\to X$ be a morphism from a normal surface $Y$.  Then we
frequently denote by $B_Y$ the $\bR$-divisor on $Y$ defined by the
identity $K_Y+B_Y = f^*(K_X+B)$.

A \emph{log resolution} is a resolution of singularities
$f\colon Y\rightarrow X$ such that $\exc(f)\cup f^{-1}_*B$ has a
simple normal crossing support, where $\exc(f)$ is the exceptional
locus of $f$ and $f^{-1}_*B$ is the strict transform of $B$ on Y. We
can write
\[
K_Y = f^*(K_X+B) +\sum_i a_iE_i.
\]
One usually writes $a_i=a(E_i, X,B)$ and calls it the \emph{discrepancy}
of $E_i$ with respect to $(X,B)$. One says that $(X, B)$ is \emph{log
  canonical} (resp.~\emph{Kawamata log terminal}) if $a_i\geq -1$
(resp.~$a_i>-1$) for every $i$. The numbers
$b(E_i, X, B):=-a(E_i, X, B)$ are called \emph{codiscrepancies}. We
use abbreviations \emph{lc} and $\emph{klt}$ for log canonical and
Kawamata log terminal respectively.
Thus, $(X,B)$ is lc, resp. klt, if the
coefficients in $B_Y$ are $\le1$, resp. are~$<1$.

A \emph{nonklt}, or \emph{nklt} center of a log surface $(X,B)$ is an image
of a divisor $E$ on $Y$ whose coefficient in $B_Y$ is $\ge1$,
i.e. $=1$ in the log canonical case. In other words, $E$ is either a
component of $B=\sum b_jB_j$ with $b_j\ge 1$ or an image of an
exceptional divisor with  $a(E, X,B)\leq -1$. This definition does not
depend on a log resolution.
The \emph{nklt locus} of $(X,B)$, denoted by $\nklt(X, B)$,
is the union of all nklt centers.

A log surface $(X, B)$ is \emph{dlt} if it is log canonical, and there exists
a finite collection of points $S\subset X$ such that
$(X\setminus S,\, \supp B\setminus S)$ is a smooth normal crossing
pair, and there are no nklt centers contained in $S$.

\begin{defn}\label{def:dlt-blowup}
  Let $(X, B = \sum b_jB_j)$ be a log surface, $b_i\ge0$. We do not
  assume that it is 
  lc and we allow some coefficients to be
  $b_i>1$, but $K_X+B$ is $\bR$-Cartier.  A \emph{dlt blowup} of
  $(X,B)$
  is a partial resolution of singularities $f\colon Y\to X$ such that
  \begin{enumerate}
  \item $Y$ is $\bQ$-factorial.
  \item The pair 
    $(Y, \sum_j \min(1,b_j)f_*\inv B_j + \sum_i E_i)$ is dlt. Here,
    $E_i$ are the exceptional divisors of $f$.
  \item One has $B_Y\ge f_*\inv B + \sum E_i$, i.e. all discrepancies
    are $a(E_i,X,B)\le -1$.
  \end{enumerate}
  If $(X,B)$ is lc then it follows that   $B_Y = f_*\inv B+\sum E_i$.
\end{defn}

It is a result of Hacon that dlt blowups exist in any dimension, see
e.g. \cite[Theorem 10.4]{fujino2011fundamental-theorems}.  For
surfaces, this is an elementary fact: a dlt blowup is obtained by
taking a log resolution followed by contracting back the exceptional
curves with discrepancy $a(E,X,B)>-1$. The configuration of such
curves is log terminal, hence rational, and the contraction exists by
\cite{artin1962some-numerical}.

\begin{defn}\label{def:effective-resolution}
  We will say that $f\colon Y\to X$ is an \emph{effective resolution}
  of a log surface $(X,B)$ if $Y$ is smooth and $B_Y\ge0$, i.e. the
  discrepancies are $a_i\le0$. For example, the minimal resolution of
 surface singularities is effective.
\end{defn}

\begin{defn}\label{def:accessible-center}
  Let $(X,B)$ be a log canonical surface, and let $Z$ be an nklt
  center (so a point or a curve). Let $f\colon Y\to X$ be a log
  resolution, and $E$ an irreducible component of $B_Y^{=1}$ on $Y$ such that $Z=f(E)$.  We say
  that $Z$ is an \emph{inaccessible} nklt center if
  $E$ is disjoint from the rest of $B_Y^{>0}$. We call $Z$ a
  \emph{accessible nklt center} otherwise.  It is easy to see that
  this definition does not depend on the choices of the log resolution $f$ and of the component $E$.
\end{defn}

\begin{rmk}
One sees easily that a log canonical surface $(X, B)$ has nonempty accessible nklt locus if and only if $\supp B_Y^{>0}$ is singular at some point of $B_Y^{=1}$.
\end{rmk}

From the classification of log canonical singularities (see
e.g. \cite{alexeev1992log-canonical-surface}) it follows that an nklt center
is inaccessible iff it is a point $p\in X$ which is a simple
elliptic singularity with $p\notin B$, or if it is a smooth curve $B_j$
which appears in $B$ with coefficient $b_j=1$, lies in the smooth part of
$X$, and is disjoint from the rest of $B$.

\begin{lem}\label{lem:sig-center-resolution}
  Let $Z$ be an accessible nklt center of a log canonical surface
  $(X,B)$. Then there exists an effective resolution of singularities
  $f\colon Y\to X$ such that $Z=f(B_1)$, $B_Y=\sum_{i=1}^n b_iB_i$, $n\ge2$,
  $b_1=1$, $b_2>0$, and $B_1\cap B_2\ne\emptyset$.
\end{lem}
\begin{proof}
  $Y$ can be taken to be the minimal resolution of a dlt
  blowup of $(X,B)$.
\end{proof}

\bigskip

A set $\mc\subset \bR$ is called a DCC (resp.~ACC) set if it satisfies
the descending (resp.~ascending) chain condition: a strictly
decreasing (resp.~increasing) subsequence of $\mc$ terminates.

\begin{defn}
  Let $\cC\subset (0,1]$. The \emph{derivative set}
  is defined as follows:
  \begin{equation}\label{eq:different}
    \cC' = \left\{
      1 - \frac{ 1- \sum n_j b_j}{m} \mid
      b_j\in\cC, \ m,n_j\in\bN,
      \ 1- \sum n_jb_j \ge0
    \right\} \cup \{1\}.
    \end{equation}
    It is easy to see that if $\cC$ is a DCC set then so is
    $\cC'$. Note that $\cC_0'=\cC_1'=\cC_2'=\cC_2$. 
\end{defn}

One standard way the derivative set appears is the adjunction formula
\cite{shokurov1992three-dimensional}. Let $(Y,  E+\Delta)\in\cS(\cC)$
be a log canonical  surface
with $E$ a reduced curve. 
In particular, $K_Y+E+\Delta$ is
an $\bR$-Cartier and the curve $E$ is at worst nodal. Then the
restriction of $K_Y+E+\Delta$ to $E$ is $K_E + \Diff_E(\Delta)$, where
$\Diff_E(\Delta) = \sum b_k' Q_k$, called \emph{the different}, is an
effective divisor with coefficients $b'_k\in\cC'$.  
The explicit formula of the coefficients $b'_k$ is as follows.

\begin{lem}\label{lem:different}
  Let $(Y, E + \sum b_j\Delta_j)$ be a log canonical surface as above.  Then
  for each~$k$
  \begin{enumerate}
  \item either $b'_k=1$,
  \item or $Q_k\in Y$ is smooth and $m=1$, $n_j= (E\cdot \Delta_j)_{Q_k}$
    in equation~\eqref{eq:different},
\item or $Q_k\in Y$ is a singularity and the exceptional locus of its
  minimal resolution $g\colon Z\to Y$ is a chain of curves
  $F_1, \dotsc, F_r$ with $F_i^2=-p_i\le 2$, and $EF_1=1$, $EF_i=0$
  for $i>1$. Then
  \begin{enumerate}
  \item $m=\det (p_1,\dotsc, p_r)$ is the determinant of the matrix $M$
    corresponding to the resolution, with $p_i$ on the main diagonal and
    $-1$ next to the main diagonal. The number $m$ is the index of $Q_k\in Y$.
  \item
    $n_j = \sum_{i=1}^r (f_*\inv \Delta_j\cdot F_i) \det(p_{i+1}, \dotsc,
    p_r)$. Here, $\det=1$ if $i=r$. 
  \end{enumerate}

  \end{enumerate}

\end{lem}
\begin{proof}
  Assume that $Q_k{\in Y}$ is singular and that $b'_k<1$. By the
  classification of log canonical surface singularities with
  coefficients in $\{1\}$  (\cite{alexeev1992log-canonical-surface}),
  the resolution graph of $Q_k\in Y$ is a chain. The log discrepancies
  $(1+f_i)$ of $K_Y+E+\Delta$ along $F_i$ are solutions to a
  $r\times r$ system of linear equations with the matrix $M$. The
  formula for $M\inv$ in terms of cofactors gives the stated formula.
  The number $\det(p_{i+1}, \dotsc, p_r)$ is the minor of $M$ obtained
  by removing the first column and $i$-th row.
\end{proof}

In particular, we have the following well known fact:
\begin{cor}
  \label{cor:diff-C2}
  For the set $\cC_2$, if $b_k' = 1-\frac1{n}<1$ then $n=mn'$, where
  $m=\det(p_1, \dotsc, p_r)$ is the index of $Q_k\in Y$, and 
  \begin{enumerate}
  \item either $\supp g_*\inv\Delta$ intersects only the last curve
    $F_r$, once, and the corresponding coefficient of $g_*\inv \Delta$
    equals $1-\frac1{n'}$;
  \item or $\supp g_*\inv\Delta$ is disjoint from $\cup F_i$, and
    $m=n$, $n'=1$.
  \end{enumerate}
  In particular, $n(K_Y+E+\Delta)$ is integral and Cartier in a
  neighborhood of $Q_k$.
\end{cor}

\bigskip

We now introduce an auxiliary ACC set. 

\begin{lem}\label{lem:two-sets}
  Let $\cC\subset(0,1]$ is a DCC set and $m\in\bN$. Then the following
  set
  \begin{eqnarray*}
    &&T_m(\cC) =  \left\{
       \frac{1-b}{\{mb\}} \mid b\in \mc,\ \{mb\}\ne0
       \right\} \cup \{1\}  
  \end{eqnarray*}
  is an ACC set and thus attains the maximum $t_m(\mc)\in\bR_{\ge1}$.
  For any $b\in\cC$ with $\{ mb \} \ne0$ one has
    $b + t_m \{mb\} \ge 1$. Here, $\{x\}$ denotes the fractional part of $x$.
\end{lem}
\begin{lem}\label{lem:t_m}
  For $\cC=\cC_0$, $\cC_1$ or $\cC_2$, one has $t_m=1$ for all $m$.
\end{lem}

The proofs of these lemmas are straightforward, and are skipped.

\section{Construction of accumulation points}\label{sec: constr}

In this section, we prove one direction of Theorem~\ref{thm: main}: if
there exists a log canonical surface $(X,B)\in\ms(\overline \mc\cup\{1\})$ with ample
$K_X+B$ satisfying the conditions (ii) and (iii) then $(K_X+B)^2$ is
an accumulation point of $\kk(\cC)$.

\subsection{Accumulation points of volumes due to accumulating coefficients}\label{subsec: constr}

\begin{thm}\label{thm: acc coe}
Let $\mc\subset (0,1]$ be a DCC set. Let $(X, B)\in \ms(\overline \mc)$ such that $K_X+B$ big and nef, and at least one coefficient of $B$ is an accumulation point of $\mc$. Then $(K_X+B)^2$ is an accumulation point of $\kk(\mc)$. 
\end{thm}
\begin{proof}
  Write $B=\sum_{j\in J} b_j B_j$. Let $J_\infty$ be the set of the
  indices $j$ such that $b_j$ is an accumulation point of $\mc$. Then
  $J_\infty$ is nonempty by assumption. Since bigness is an open
  condition on the $N^1(X)_\R$, the space of numerical classes of
  $\R$-divisors on~$X$, there is a strictly increasing sequence
  $b_j^{(s)}\in \mc$ converging to $b_j$ for each $j\in J_\infty$ such
  that $K_X+B^{(s)}$ is big, where
\[
B^{(s)}:=\sum_{j\in J_\infty} b_j^{(s)}B_j +\sum_{j\in J-J_\infty} b_j B_j.
\]
Since $(X, B)$ is log canonical and $0\leq B^{(s)}< B$ for any $s$,
the log surfaces $(X, B^{(s)})$ are all log canonical. By construction,
$(X, B^{(s)})\in\ms(\mc)$.
Since $K_X + B$ is nef, one has $\vol(K_X+B^{(s)}) < \vol(K_X+B)$. 
Finally, since the volume is a continuous function on $N^1(X)_\R$, we have
\[
\lim_{s\rightarrow\infty} \vol(K_X+B^{(s)}) = \vol(K_X+ B) =(K_X+B)^2.
\]
\end{proof}

\subsection{Accumulation points of volumes due to nklt loci}\label{sec: nklt}

\begin{defn}
  We denote by $\mc_B$
  the set of coefficients appearing in $B$.
\end{defn}

\begin{thm}\label{thm: nklt}
  Let $(X, B)$ be a log canonical surface with ample $K_X+B$ and
  nonempty accessible nklt locus.  For any $\epsilon>0$, there is a
  log canonical surface $(X', B')\in\ms(\mc_B)$ with ample $K_{X'}+B'$
  such that
\begin{enumerate}
\item $X'$ birational to $X$ and $B'$ is the strict transform of $B$, and
\item $\vol(K_X+B)-\epsilon < \vol(K_{X'}+B') <\vol(K_X+B)$.
\end{enumerate} 
\end{thm}

\begin{rmk}
The construction appeared in \cite{Liu2017minimal-volume} when the coefficient set is $\mc_1$.
\end{rmk}
\begin{proof}
  Consider a resolution as in Lemma~\ref{lem:sig-center-resolution}
  and $K_Y+B_Y = f^*(K_{ X}+B)$.  By blowing up a point of
  intersection of $B_1$ and $B_2$ several times, if necessary, we can
  assume that $f$ is effective, and $B_1,B_2$ meet transversally at a
  single point $p\in Y$.

  In our notation, $b_1=1$ and $b_2>0$.
  Let us blow up $p\in Y$ and then
  its preimage on the strict transforms of $B_2$ on the blown-up
  surfaces. Let $h\colon Y^{(s)}\rightarrow Y$ be the resulting morphism
  after $s$ blow-ups. 
  The inverse image of $B_1+B_2$ on $Y^{(s)}$ has the
  following dual graph:
\begin{center}
  \begin{tikzpicture}
    \begin{scope}
           \node [inner sep=0pt](a0) at (0,0) {\tiny $\otimes$};
                      \node [inner sep=0pt](a5) at (5,0) {\tiny $\otimes$};
      \node[draw, inner sep=1pt,fill=white,circle] (a1) at (1,0) {};
      \node[draw, inner sep=1pt,fill=white,circle] (a2) at (2,0) {};
            \node[draw, inner sep=1pt,fill=white,circle] (a3) at (3,0) {};
      \node[draw, inner sep=1pt,fill=white,circle] (a4) at (4,0) {};
      \node[below of=a0, node distance=1em] {\tiny$B_1^{(s)}$};
           \node[below of=a5, node distance=1em] {\tiny$B_2^{(s)}$};
      \node[above of=a1, node distance=1em]{\tiny(1)};
      \node[below of=a1, node distance=1em]{\tiny$E_s$};
      \node[above of=a2, node distance=1em]{\tiny(2)};
      \node[below of=a2, node distance=1em]{\tiny$E_{s-1}$};
      \node[above of=a3,node distance=1em]{\tiny(2)};
      \node[below of=a3,node distance=1em]{\tiny$E_{2}$};
      \node[above of=a4,node distance=1em]{\tiny(2)};
      \node[below of=a4,node distance=1em]{\tiny$E_1$};
      \draw (a0)--(a1)--(a2);
      \draw (a3)--(a4)--(a5);
      \draw[dashed] (a2)--(a3);
            \end{scope}
  \end{tikzpicture}
  \end{center}
  where the $B_j^{(s)}$'s are the strict transforms of the $B_j$'s
  ($j=1,2$), the $E_i$'s ($1\leq i\leq n$) are the exceptional curves,
  and the numbers above the nodes are the negatives of the
  self-intersections of the corresponding curves.

Let $K_{Y^{(s)}} + B_{Y^{(s)}}= h^*(K_Y+B_Y)$.  If we write
$B_{Y^{(s)}} = h^{-1}_*B_Y + \sum_{1\leq i\leq n} e_iE_i$ as the sum of
distinct components, then $e_1 = e_2=\dots=e_n=b_2>0$.
Now let
$$B_{Y^{(s)}}'=B_{Y^{(s)}}-b_2\sum_{1\leq i\leq s}\frac{i}{s} E_i=  h^{-1}_*B_Y +b_2\sum_{1\leq i \leq s} \frac{s-i}{s}E_i.$$
Since $K_{Y^{(s)}}+B_{Y^{(s)}}$ is big and nef and $B_{Y^{(s)}}'< B_{Y^{(s)}}$, one has 
\[
\vol(K_{Y^{(s)}}+B_{Y^{(s)}}')<\vol(K_{Y^{(s)}}+B_{Y^{(s)}}).
\]
On the other hand, one computes
\[
(K_{Y^{(s)}}+B_{Y^{(s)}}')^2=(K_{Y^{(s)}}+B_{Y^{(s)}})^2 -\frac{b_2^2}{s}=(K_X+B)^2-\frac{b_2^2}{s}
\]
Take $s>\max\{ \frac{b_2^2}{\epsilon}, \frac{b_2^2}{(K_X+B)^2}\}$. Then
\[
  \max\{0,  (K_X+B)^2-\epsilon  \}< (K_{Y^{(s)}}+B_{Y^{(s)}}')^2 < (K_X+B)^2.
\]
Since $h_*(K_{Y^{(s)}}+B_{Y^{(s)}}')= K_Y+B_Y$ is big, we infer that,
$K_{Y^{(s)}}+B_{Y^{(s)}}'$
is big.  Let $(X^{(s)}, B^{(s)})$ be its log canonical model.
 
One computes easily that $(K_{Y^{(s)}}+B_{Y^{(s)}}')E_i=0$ for
$1\leq i\leq s-1$.  Thus, all the exceptional curves $F\ne E_s$ of $Y^{(s)}\to X$
satisfy $(K_{Y^{(s)}}+B_{Y^{(s)}}')F\le 0$. Obviously, they
are all contracted by the morphism $g\colon Y^{(s)}\to X^{(s)}$.
 Since the coefficient of $E_s$ in $B'_{Y^{(s)}}$ is zero, one has
 $(X^{(s)}, B_{X^{(s)}})\in \ms(\mc_B)$. Moreover,
$$(K_{X^{(s)}}+B^{(s)})^2\geq(K_{Y^{(s)}}+B_{Y^{(s)}}')^2=
(K_X+B)^2-\frac{b_2^2}{s}>(K_X+B)^2 - \epsilon.$$ We take
$(X', B')=(X^{(s)}, B^{(s)})$.
\end{proof}

\begin{cor}\label{cor: nklt1}
  Let $(X, B)$ be a log canonical surface with ample $K_X+B$ and
  nonempty accessible nklt locus. For any $\epsilon>0$, there is a
  log canonical surface $(X', B')\in \ms(\mc_B)$ with ample $K_{X'}+B'$
  such that
\begin{enumerate}
\item $X'$ is birational to $X$ and $B'$ is the strict transform of $B$;
\item $(X', B')$ has no accessible nklt locus;
\item $\vol(K_X+B)-\epsilon < \vol(K_{X'}+B') <\vol(K_X+B)$.
\end{enumerate}
\end{cor}
\begin{proof}
  By Theorem~\ref{thm: nklt} there is a log surface
  $(X^{(1)}, B^{(1)})\in\ms(\mc_B)$ with ample $K_{X^{(1)}}+B^{(1)}$
  such that $X^{(1)}$ is birational to $X$ and $B^{(1)}$ is the strict
  transform of $B$, and
  $$\vol(K_X+B)-\frac{\epsilon}{2} < \vol(K_{X^{(1)}}+B^{(1)}) <\vol(K_X+B).$$
  If $(X^{(1)},B^{(1)})$ has nonempty accessible nklt locus, then we can apply
  Theorem~\ref{thm: nklt} again to $(X^{(1)}, B^{(1)})$ to  get
  $(X^{(2)}, B^{(2)})$.  By induction, if we
  have constructed $(X^{(s)}, B^{(s)})$ in $\ms(\mc_B)$, and if
  $(X^{(s)}, B^{(s)})$ has nonempty accessible nklt locus, then we
  can apply Theorem~\ref{thm: nklt} to $(X^{(s)}, B^{(s)})$ to obtain
  a log surface $(X^{(s+1)}, B^{(s+1)})\in\ms(\mc_B)$ with ample
  $K_{X^{(s+1)}}+B^{(s+1)}$ such
  that
  $$\vol(K_{X^{(s)}}+B^{(s)})-\frac{\epsilon}{2^{s+1}} <
  \vol(K_{X^{(s+1)}}+B^{(s+1)}) <\vol(K_{X^{(s)}}+B^{(s)}).$$

  Since $\kk(\mc_B)$ is a DCC set, this process must stop after, say,
  $N-1$ steps.
  Then $(X^{(N)}, B^{(N)})$ has
  ample $K_{X^{(N)}}+B^{(N)}$ and empty accessible nklt locus, and
  it is birational to $X$ and $B^{(N)}$ is the strict transform of
  $B$. Moreover,
$$(K_X+B)^2-\epsilon < (K_X+B)^2-\sum_{s=1}^N \frac{1}{2^s}\epsilon<(K_{X^{(N)}}+ B^{(N)})^2 <(K_X+B)^2. $$
Now we take $(X', B')=(X^{(N)}, B^{(N)}) $.
\end{proof}

\begin{cor}\label{cor: nklt2}
  Let $(X, B)$ be a log canonical surface with ample $K_X+B$ and
  nonempty accessible nklt locus. Then there is an infinite
  sequence of log canonical surfaces $(X^{(s)}, B^{(s)})\in \ms(\mc_B)$ with ample $K_{X^{(s)}}+ B^{(s)}$
  such that
\begin{enumerate}
\item $X^{(s)}$ is birational to $X$ and $B^{(s)}$ is the strict transform of $B$;
\item $(X^{(s)}, B^{(s)})$ has no accessible nklt locus;
\item the volumes $\vol(K_{X^{(s)}}+B^{(s)})$ are strictly increasing and 
\[
\vol(K_X+B)=\lim_{s\rightarrow\infty} \vol(K_{X^{(s)}}+B^{(s)}).
\]
\end{enumerate}
\end{cor}

\subsection{The if part of Theorem~\ref{thm: main}}
\begin{proof}[Proof of Theorem~\ref{thm: main}: the if part]
  First we assume that $1\in\overline \mc$, so
  $\overline \mc \cup\{1\}=\overline \mc$. If $(X, B)$ satisfies
  (ii.a), then there is a birational morphism $f\colon Y\rightarrow X$
  extracting a single curve, with codiscrepancy contained in
  $\acc(\mc)$. As usual, we write $K_B+B_Y = f^*(K_X+B)$. Note that $B_Y$ is
  effective and the log surface $(Y, B_Y)$ satisfies the conditions of
  Theorem~\ref{thm: acc coe}, so we have
    \begin{math}
      (K_X+B)^2 = (K_Y+B_Y)^2\in\acc(\kk(\mc)).
    \end{math}
  If $(X,B)$ satisfies (ii.b), then by  Corollary~\ref{cor: nklt2}
  \begin{displaymath}
\vol(K_X+B)\in\acc(\kk(\mc_B)).    
  \end{displaymath}
In this case, we can assume that $(X,B)$ does not satisfy (ii.a), so
$\mc_B\subset\mc$ and hence $\vol(K_X+B)\in \acc(\kk(\mc))$.

Now we assume that $1\notin\overline \mc$.
Then each irreducible component of $\lfloor B\rfloor$ has geometric genus at
most 1 by assumption (iii).  By Corollary~\ref{cor: nklt2}, we get a
sequence $(X^{(s)}, B^{(s)})\in\ms(\mc\cup\{1\})$
with volumes converging to $(K_X+B)^2$
but with an additional
property: each component of $\lfloor B^{(s)} \rfloor$ has
geometric genus $\le1$. But since there is no accessible nklt locus,
each component of $\lfloor B^{(s)} \rfloor$ must be a smooth curve
lying in the smooth locus of $X^{(s)}$ and disjoint from the rest of
$B^{(s)}$. By adjunction, it then must have genus $\ge2$. Thus,
$\lfloor B^{(s)} \rfloor=0$ and $(X^{(s)}, B^{(s)})\in\ms(\mc)$.

\end{proof}

\begin{ex}\label{ex: Z}
First consider the log canonical surface $(\PP^1\times \PP^1, B)$, where $B$ is a reduced divisor consisting $3$ horizontal lines $L_i$ and $n\geq 3$ vertical lines $V_j$. Then $K_{\PP^1\times\PP^1}+B$ is ample and $(K_{\PP^1\times\PP^1}+B)^2= 2(n-2)$, thus exhausting all the positive even integers as $n$ varies. By Theorem~\ref{thm: main}, each of the numbers $2(n-2)$ is an accumulation point of $\kk(\mc_0)$.

Now let $\FF_1$ is the Hirzebruch surface obtained by blowing up one
point of $\PP^2$ and let
$B=\Gamma_0 + \Gamma_1+\Gamma_2 + \sum_{1\leq j\leq n} F_j$
($n\geq 2$), where $\Gamma_0$ is the section with $\Gamma_0^2=-1$,
$\Gamma_1$ and $\Gamma_2$ are two distinguished sections with
$\Gamma_1^2=\Gamma_2^2=1$, and the $F_j$'s are distinct fibers 
  not passing through $\Gamma_1\cap \Gamma_2$. Then the pair
$(\FF_1, B)$ is log canonical, has ample $K_{\FF_1}+B$ and
$(K_{\FF_1}+B)^2=2n-3 $, which exhausts all the positive odd integers
as $n$ varies. By Theorem~\ref{thm: main}, each of $2n-3$ with
$n\geq 2$ is an accumulation point of $\kk(\mc_0)$.

As a result of the above construction, we obtain $\Z_{>0}\subset\acc(\kk(\mc_0))$, which proves one part of \cite[Conjecture 3, (a)]{blache1995example-concerning}.
\end{ex}


\section{The limit surface}\label{sec: limit surf}
\begin{lem}\label{lem: fin bl}
  Let $Z$ be a projective normal surface and $B=\sum_j b_jB_j$ an
  effective $\RR$-Cartier divisor on $Z$ with $b_0<1$. Then there are
  finitely many effective resolutions $g_t\colon Z_t\rightarrow Z$,
  $1\leq t\leq r$ such that, if $g\colon Y\rightarrow Z$ is an
  effective resolution, then there is an $t$ such that the birational
  map $g_t^{-1}\circ g\colon Y\dashrightarrow Z_t $ is a morphism and
  it is an isomorphism over a Zariski open neighborhood of the strict
  transform $g_{t*}^{-1} B_0$ of $B_0$.
\end{lem}
\begin{proof}
  Since any resolution $f\colon Y\rightarrow Z$ dominates the minimal
  resolution of $Z$, we can replace $Z$ by its minimal resolution and
  the log canonical divisor $K_Z+B$ by its
  pull-back. In other words, we can assume that $Z$ is smooth.

  We will take $g_t\colon Z_t\rightarrow Z$ to be effective
  resolutions (cf. \eqref{def:effective-resolution})
  that only blow up
  points on $B_0$ and its strict transforms. The lemma is equivalent to
  the statement that there are only finitely many such effective
  resolutions.

  Suppose that $f\colon Y\rightarrow Z$ is such a resolution. Then
  $f$ cannot blow up any smooth point of $\supp B$ that lies in $B_0$,
  since $b_0<1$. So there are only finitely many choices for the first
  level of blow-ups at the points of $B_0$.

  Over every point $p\in B_0$ that is blown up, there are finitely
  many blow-ups that make the support of the total transform of $B$
  becomes simple normal crossing over~$p$.  Let $p_1, \dots, p_k $ be
  the inverse images of $p$ lying on the strict transform of
  $B_0$. Then the codiscrepancy of the exceptional divisor of the
  $k$-th blowup over $p_i$ and at the strict transform of $B_0$ takes
  the form $a_i - k(1-b_0)$, where $a_i$ is the coefficient of the
  (unique) boundary component in the pull-back of $K_Z+B$ that
  intersects the strict transform of $B_0$ at $p_i$. Under the
  condition that $b_0<1$ and $a_i - k(1-b_0)\geq 0$ there are only
  finitely many further blow-ups over $p_i$.
\end{proof}

Let $\mc\subset (0,1]$. We define the following subset of $\kk(\mc)$ 
\[
\kk_{\klt}(\mc) := \{(K_X+B)^2\mid (X, B)\in \ms(\mc)\, \klt, K_X+B \text{ ample }\}
\]
\begin{lem}\label{lem: key}
Let $\mc\subset(0,1]$ be a DCC set. Let $v_\infty \in \acc(\kk(\mc))$. Then there are a smooth projective surface $(Z, D)$ with $D$ a reduced  divisor, and a sequence of smooth projective surfaces $Y^{(s)}$ with a boundary divisor $B_{Y^{(s)}}$ such that 
\begin{enumerate}
\item there is a birational morphism $g_s\colon Y^{(s)}\rightarrow Z$ which only blows up the nodes of $D$ and of its the total transforms;
\item $(Y^{(s)}, B_{Y^{(s)}})$ is log canonical, $K_{Y^{(s)}}+B_{Y^{(s)}}$ is big and nef, $g_{s*} B_{Y^{(s)}}\leq D$;
\item $\vol(K_{Y^{(s)}}+ B_{Y^{(s)}})$ form a strictly increasing sequence, converging to $v_\infty$;
\item the log canonical model $(X^{(s)}, B^{(s)})$ of $(Y^{(s)}, B_{Y^{(s)}})$ lies in $\ms(\mc)$;
\item there are no ($-1$)-curves $E$ on $Y^{(s)}$ such that $(K_{Y^{(s)}}+B_{Y^{(s)}})E=0$, that is, $Y^{(s)}\rightarrow X^{(s)}$ is the minimal resolution;
\item if $v_\infty \in \kk_{\klt}(\mc) $ then, additionally, $(X^{(s)}, B^{(s)})$ is klt and it holds
\[
g_t^*(K_Z+g_{s*} B_{Y^{(s)}}) \leq K_{Y^{(t)}} + B_{Y^{(t)}}
\]
for any $s<t$.
\end{enumerate}
\end{lem}

\begin{proof}
Since $\kk(\mc)$ is a DCC set (\cite[Theorem 8.2]{alexeev1994boundedness-and-ksp-2}), $v_\infty$
is approached by a strictly increasing sequence $(K_{X^{(s)}}+B^{(s)})^2$
in $\kk(\mc)$, where the $(X^{(s)}, B^{(s)})$'s are surfaces in $\ms(\mc)$
whose log canonical divisors $K_{X^{(s)}}+B^{(s)}$ are ample.

By \cite[Theorem~7.6]{alexeev1994boundedness-and-ksp-2} (see also \cite[Theorem~4.7]{alexeev2004bounding-singular}),
there is a diagram
\[
\begin{tikzcd}
Y^{(s)} \ar[d, "f_s"'] \ar[r, "g_s"]& Z^{(s)}\\
X^{(s)} &
\end{tikzcd}
\]
such that 
\begin{enumerate}
\item[(a)] $f_s\colon Y^{(s)}\rightarrow X^{(s)}$ is the minimal resolution of
   $(X^{(s)}, B^{(s)})$,  
\item[(b)] $g_s\colon Y^{(s)}\rightarrow Z^{(s)}$ is birational, and, defining
  $K_{Y^{(s)}}+B_{Y^{(s)}} = f_s^*(K_{X^{(s)}}+B^{(s)})$,
  $D^{(s)}=g_s(\supp B_{Y^{(s)}}\cup\exc(f_s))$, the log surfaces $(Z^{(s)}, D^{(s)})$ form a
  bounded class.
\end{enumerate}
Since the log surfaces $(Z^{(s)}, D^{(s)})$ form a bounded class, it suffices
to assume that $(Z^{(s)}, D^{(s)})$ are fixed, to be denoted by
$(Z, D)$ (\cite[Remark~5.7]{alexeev2004bounding-singular}). By
\cite[Lemma 5.5 and 5.6]{alexeev2004bounding-singular}, we can assume
that $Z$ is smooth and $g_s\colon Y^{(s)}\rightarrow Z$ blows up only
nodes of $D$ and of its total transforms.

Thus we have found $(Z, D)$ and $(Y^{(s)}, B_{Y^{(s)}})$ satisfying
(i)-(v). The last very important property (vi) is
\cite[Theorem~8.5]{alexeev1994boundedness-and-ksp-2} (see also
\cite[Theorem~5.8]{alexeev2004bounding-singular}).
\end{proof}

Now we are ready to complete the proof of Theorem~\ref{thm: main}.
\begin{proof}[Proof of the only if part of Theorem~\ref{thm: main}]
Let $v_\infty \in \acc(\kk(\mc))$. Let $(Z, D)$, $(Y^{(s)}, B_{Y^{(s)}})$ and $(X^{(s)}, B^{(s)})$ be as in Lemma~\ref{lem: key}.

\medskip

\noindent{\bf klt case.} First we assume that $v_\infty \in \acc(\kk_\klt(\mc))$. Then by Lemma~\ref{lem: key}, $(X^{(s)}, B^{(s)})$ have klt
singularities for all $s$ and for every $s<t$,
\begin{equation}\label{eq: key}
g_t^*(K_Z+B_Z^{(s)}) \leq K_{Y^{(t)}} + B_{Y^{(t)}}
\end{equation}
where the divisor $B_Z^{(s)}=g_{s*}(B_{Y^{(s)}})$. In particular, it follows
that $B_Z^{(s)}\leq B_Z^{(t)}$ for any $s<t$.

Let $B_Z=\lim_{s\to\infty} B_Z^{(s)}$. Since
$K_Z+B_Z^{(s)}$ is nef, so is their limit $K_Z+B_Z$. Moreover, the
bigness of $K_Z+B_Z^{(s)}$ implies that of $K_Z+B$ since
$K_Z+B_Z^{(s)}\leq K_Z+B$. By \eqref{eq: key} one has
$(K_{Y^{(s)}} + B_{Y^{(s)}})^2\leq (K_Z+B_Z^{(s)})^2\leq (K_{Y^{(t)}} + B_{Y^{(t)}})^2$,
one sees that
$$(K_Z+B_Z)^2=\lim_{s\to\infty} (K_Z+B_Z^{(s)})^2 = \lim_{s\to\infty}(K_{X^{(s)}}+B^{(s)})^2=v_\infty.$$ 

We need to show that $(Z, B_Z)$ is log canonical. Obviously, the
divisor $B_Z$ has coefficients in $(0,1]$. Let $E$ be a divisor over
$Z$. By \eqref{eq: key} we have, for every $s<t$,
$$b(E, Z, B_Z^{(s)})\leq b(E, Y^{(t)}, B_{Y^{(t)}})\leq 1,$$ where $b$ denotes
the codiscrepancies of the log surfaces. It follows that
$$b(E, Z, B_Z)=\lim_{s\to\infty} b(E, Z, B_Z^{(s)})\leq b(E, Y^{(t)},B_{Y^{(t)}})\leq 1.$$
Hence $(Z, B_Z)$ is log canonical.

However, it is possible that the coefficient set $\mc_{B_Z}$ is not contained in $\overline \mc\cup\{1\}$, as required. We need to contract all the components of $B_Z$ that have coefficients not in $\overline \mc\cup\{1\}$. For this, we consider $(Z, B_Z)\rightarrow (Z_\can, B_\can)=:(X, B)$, the contraction onto the log canonical model. We will show that $(X, B)\in\ms(\overline \mc\cup\{1\})$ and it satisfies the properties (i), (ii), (iii) of Theorem~\ref{thm: main}.

\medskip

\noindent{\bf Claim.} $(X, B)\in \ms(\overline\mc\cup\{1\})$.
\begin{proof}[Proof of the claim.] 
Being the log canonical model of $(Z, B_Z)$, the surface $(X, B)$ is log canonical. We need to show that $\mc_B\subset\overline \mc\cup\{1\}$, or equivalently, any component $B_j$ of $B_Z$ with coefficient $b_j\notin \overline \mc\cup\{1\}$ is
contracted by the morphism $Z\rightarrow X$.

Since $b_j<1$, up to
taking a subsequence, there is a resolution $h\colon Z'\rightarrow Z$ such that for each $s$, the birational map $g_s'=h^{-1}\circ g_s\colon Y^{(s)}\dashrightarrow Z'$ is a morphism and is an isomorphism over a Zariski open neighborhood of $h_*^{-1} B_j$ by Lemma~\ref{lem: fin bl}. 

 By taking a subsequence, we can assume that,
for any $s$, $b(B_j, K_{Y^{(s)}}, B_{Y^{(s)}})\notin \mc$, which implies that
$B_j$ is contracted by the morphism $Y^{(s)}\rightarrow X^{(s)}$ and hence
$(K_{Y^{(s)}}+ B_{Y^{(s)}})B_{j, Y^{(s)}}=0$, where $B_{j, Y^{(s)}}$ is the strict
transform of $B_j$ on $Y^{(s)}$. 

We compute $  (K_Z+B_Z)B_j$ on $Z'$. Note that, for any fixed curve $E$ on $Z'$, the codiscrepancies $b(E, Y^{(s)}, B_{Y^{(s)}})$ converge to $b(E, Z, B_Z)$. For the computation of $  (K_Z+B_Z)B_j $, only the codiscrepancies of those finitely many curves in $h^{-1}(D)$ that intersect $B_j$ are relevant. Thus, we have
\begin{multline*}
  (K_Z+B_Z)B_j = h^*(K_Z+B_Z) h_*^{-1} B_j \\= \lim_{s\rightarrow\infty} g_{s*}'(K_{Y^{(s)}}+ B_{Y^{(s)}}) h_*^{-1} B_j =
  \lim_{s\rightarrow\infty} (K_{Y^{(s)}}+ B_{Y^{(s)}})B_{j, Y^{(s)}}=0.
\end{multline*}
\end{proof}

\medskip

\noindent{\bf Claim.}  $(X, B)$ satisfies condition (ii) of Theorem~\ref{thm: main}.

\begin{proof}[Proof of the claim.]
Suppose that $(X, B)$ does not satisfy condition (ii.b). We need to show that  $(X, B)$ satisfies (ii.a). 

Since $(X, B)$ does not satisfy condition (ii.b), $\supp B_Z$ is
smooth along $\nklt(Z, B_Z)$ and so $g_s\colon Y^{(s)}\rightarrow Z$
blow up only at the klt points of $(Z, B_Z)$ by Lemma~\ref{lem:
  key}(i). By Lemma~\ref{lem: fin bl} there are only finitely many
possible blow-ups $f\colon Y\rightarrow Z$ such that the pull-back
$f^*(K_Z+B_Z)=K_Y+B_Y$ is still a log canonical divisor with an
effective boundary $B_Y$. Since $B_{Y^{(s)}}$ is effective and
$g_{s*}(B_{Y^{(s)}})\leq B_Z$, the morphisms $Y^{(s)}\rightarrow Z$
must be among these finitely many blow-ups. Thus, up to taking a
subsequence, we can assume that $Y^{(s)}=Y$ and $\supp B_{Y^{(s)}}$
are the same for all $s$.

Since $(K_{Y^{(s)}}+B_{Y^{(s)}})^2$ is strictly increasing, up to taking a subsequence, the boundary divisors $B_{Y^{(s)}}$ form a strictly
increasing sequence. There must be a component of $\supp B_{Y^{(s)}}$, say $B_0$,
that is not contracted by $f_s\colon Y^{(s)}\rightarrow X^{(s)}$ for any $s$, and
$b(B_0, Y^{(s)}, B_{Y^{(s)}})$ is strictly increasing and converging to
$b(B_0, Z, B_Z)$, the coefficient of $B_0$ in $B_Z$. Since $B_0$ is
not contracted by $f_s\colon Y^{(s)}\rightarrow X^{(s)}$ and $\mc_{B^{(s)}}\subset\mc$, its coefficient in $B_{Y^{(s)}}$ is
in $\mc$. This means $b(B_0, Z, B_Z)$ is an accumulation point of
$\mc$, and condition (ii,a) of Theorem~\ref{thm: main} is verified.
\end{proof}

\medskip

\noindent{\bf Claim.} If $1\notin\overline \mc$, then $(X, B)$ satisfies  (iii) of Theorem~\ref{thm: main}.

\begin{proof}[Proof of the claim.]
Suppose $\lfloor B\rfloor\neq 0$ and $B_0$ is a component of $\lfloor B\rfloor$. Let $B_{0,Y^{(s)}}$ be the strict transform of $B_0$ on $Y^{(s)}$. Then 
\[
\lim_{s\rightarrow\infty} b(B_{0, Y^{(s)}}, Y^{(s)}, B_{Y^{(s)}}) = 1
\]
Since $1\notin\overline \mc$, for sufficiently large $s$, $b(B_{0, Y^{(s)}}, Y^{(s)}, B_{Y^{(s)}})$ does not lie in $\mc$. On the other hand, the log canonical model $(X^{(s)}, B_s)$ lies in $\ms(\mc)$. It can only happen that $B_{0, Y^{(s)}}$ is contracted by $f_s\colon Y^{(s)}\rightarrow X^{(s)}$ to an lc singularity, and hence the geometric genus of $B_0$ (and $B_{0, Y^{(s)}}$) is at most one.
\end{proof}

In conclusion, the theorem is proved in the case when  $v_\infty \in
\acc(\kk_\klt(\mc))$. 

\medskip 

\noindent{\bf nklt case.}  Suppose that $v_\infty\in \acc(\kk(\mc))-\acc(\kk_\klt(\mc))$. Then there are at most finitely many $s$
such that $(X^{(s)}, B_s)$ is klt. By taking an
infinite subsequence, we can assume that
$\nklt(X^{(s)}, B^{(s)})\neq \emptyset$ for all $s$. 

For each $s$, let
$\mu_s\colon\bar X^{(s)}\rightarrow X^{(s)}$ be a $\Q$-factorial dlt
blowup, extracting only divisors with discrepancy $-1$. Let
$K_{\bar X^{(s)}}+B_{\bar X^{(s)}} = \mu_s^*(K_{X^{(s)}}+B^{(s)})$, and write
$B_{\bar X^{(s)}}=B_{\bar X^{(s)}}^{(1)}+B_{\bar X^{(s)}}^{(2)}$ where
$B_{\bar X^{(s)}}^{(1)}$ is the fractional part of $B_{\bar X^{(s)}}$ and
$B_{\bar X^{(s)}}^{(2)}$ the (nonzero) reduced part.

We can choose a strictly decreasing sequence $\epsilon_s\in (0,1]$
with $\lim_{s\to\infty} \epsilon_s=0$ such that
\begin{itemize}
\item $1-\epsilon_s\notin \mc$;
\item $K_{\bar X^{(s)}}+B_{\bar X^{(s)}}'$ is lc and big,
  where
  $B_{\bar X^{(s)}}'= B_{\bar X^{(s)}}^{(1)} + (1-\epsilon_s) B_{\bar
    X^{(s)}}^{(2)}$;
\item $(K_{\bar X^{(s)}}+B_{\bar X^{(s)}}')^2$ is strictly increasing and
  $\lim_{s\to\infty} (K_{\bar X^{(s)}}+B_{\bar X^{(s)}}')^2=v_\infty$.
\end{itemize}
Let $\cD=\mc\cup\{1-\epsilon_s\}_s$.
Then $\cD$ is a DCC set and
$(\bar X^{(s)}, B_{\bar X^{(s)}}')\in \ms(\cD)$ are klt. As in the klt case,
we can then construct a log canonical surface $(Z, B_Z)$ with big and nef
$K_Z+B_Z$ whose log canonical model is in
$ \ms(\overline\cD\cup\{1\})$ .

We need to show that the log canonical model $(X, B)$ of $(Z, B_Z)$ is
actually in $\ms(\overline\mc\cup\{1\})$. By construction, each
coefficient $b$ of the divisor $B$ is either in $\acc(\cD)\cup\{1\}$
or it appears in the divisors $B_{\bar X^{(s)}}'$ 
infinitely many times. Since
$\acc(\cD)=\acc(\cC)\cup \{1\}$ and each of the coefficients
$1-\epsilon_s$ appears only once, we have
$b\in \overline\cC\cup\{1\}$.


Similar arguments as in the klt case show that $(X, B)$ satisfies the
conditions (ii) and (iii) of Theorem~\ref{thm: main}.

\end{proof}

\section{Lower bounds of accumulation points for $\kk(\cC_2)$}
\label{sec: low bnd}
Let $\mc\subset (0,1]$ be a DCC subset. By
\cite{alexeev1994boundedness-and-ksp-2} the set $\kk(\cC)$ and so also
$\acc(\kk(\cC))$ are DCC sets. The paper
\cite{alexeev2004bounding-singular} gives an effectively computable
lower bound for $\kk(\cC)$ but, as we mentioned, it is way too small
to be useful. However, we have the following quite efficient bound:

\begin{defn}\label{defn:v1}
$$v_1(\mc)=\min\{(K_X+B)^2\mid (X, B)\in\ms(\mc\cup\{1\}), K_X+B
\text{ ample and } \lfloor B\rfloor\neq 0\}.$$  
\end{defn}

\begin{thm}[Koll\'ar \cite{kollar1994log-surfaces}, (6.2.1), (5.3.1)]
  \label{thm:v1}
  One has $v_1(\cC_2) = \frac1{1764} =
  \frac1{42^2}$. 
\end{thm}

By Corollary~\ref{cor: =},
$\min\acc(\kk(\cC_2))= \min \kk_\nklt(\cC_2)$, where for the latter we
only consider log surfaces $(X,B)$ with $\nklt(X,B)\ne\emptyset$. If
$\lfloor B\rfloor\ne0$ then we are done by \eqref{thm:v1}. Below, we
deal with the remaining case: $\lfloor B\rfloor = \emptyset$ but
$(X,B)$ has an isolated nklt center, a point $p\in X$. The number
$t_m(\cC)$ in the theorem below was defined in
Section~\ref{sec:preliminaries}. 

\begin{thm}\label{thm: low bnd}
  Let $(X,B=\sum b_jB_j)$ be a log canonical surface with
  coefficients in a DCC set $\mc$. Suppose that $K_X+B$ is ample
  and that 
  $H^0(X, \lfloor m(K_X+B) \rfloor)\ne 0$. Then
  \begin{displaymath}
    \vol(K_X+B) \ge \frac{v_1(\mc)}{(1+mt_m(\cC))^2}.
  \end{displaymath}
\end{thm}

\begin{proof}
  Let $C\sim \lfloor m(K_X+B) \rfloor$ be an effective curve. Hence,
  we have the relation $m(K_X+B) \sim_\R C + \{mB\}$.  Now define
  \begin{displaymath}
    D = C + \max(B,\lceil \{mB\} \rceil).    
  \end{displaymath}

  If $B=\sum b_jB_j$ then the coefficient of $B_j$ in
  $\max(B, \lceil \{mB\} \rceil)$ is $b_j$ if $\{mb\}=0$ and~$1$ if
  $\{mb\}\ne0$.
  Thus, $K_X+D\ge K_X + B$, and the integral part of $D$ satisfies
  $\lfloor D\rfloor \ge m(K_X+B)$, so $\lfloor D\rfloor$ supports a big divisor.
  Note also that $K_X+D$ is $\bR$-Cartier because $B$ does not pass through
  the non $\bQ$-factorial singularities of $X$ since $(X,B)$ is lc.
  
  We want to bound the divisor $K_X + D$ from above in terms of
  $K_X+B$.  We search for a number $t$ such that
  \begin{eqnarray*}
    & (1 + mt)(K_X + B) \ge K_X + D \iff \\
    & B + t( C + \{mB\} ) \ge C + \max(B, \lceil \{mB\} \rceil)
  \end{eqnarray*}
  This is true for as long as $t\ge1$ and for every $b\in\cC$ with
  $\{mb\}\ne0$ one has $b + t \{mb\} \ge 1$. By
  Lemma~\ref{lem:two-sets} this holds if we set $t=t_m(\cC)$.  This
  gives us
  \begin{equation}\label{eq:vols}
    (1+mt_m)^2 \vol (K_X+B) \ge \vol (K_X+D). 
  \end{equation}

  The log surface $(X,D)$ may not be log canonical.
  Let $f\colon Y\to X$ be a dlt blowup as in Definition
  \ref{def:dlt-blowup}. Let $K_Y+D_Y = f^*(K_X+D)$.
  If $D_Y = \sum d_jD_j + \sum e_iE_i$ then one has $e_i\ge1$. Define 
  $D'_Y = \sum \min(1,d_j)D_j + \sum \min(1,e_i)E_i$. Then we have  
  \begin{displaymath}   
   f^*(K_X + D)  \ge K_Y + D'_Y \ge f^*(K_X+B), 
  \end{displaymath}
  the latter being true because $K_X+D \ge K_X + B$ and $(X,B)$ is
  log canonical. 

  This implies that the divisor $K_Y+D'_Y$ is big. Its reduced part
  $\lfloor D'_Y\rfloor = f^{-1} (\lfloor D\rfloor)$ supports a big
  divisor. Thus, on the log canonical model $(Z, D_Z)$ of $(Y, D'_Y)$ the
  reduced part $\lfloor D_Z\rfloor \ne0$. We have
  \begin{math}
    \vol (K_Y + D'_Y) = \vol (K_Z + D_Z) 
  \end{math}
  and by Theorem~\ref{thm:v1} $\vol (K_Z + D_Z) \ge v_1(\mc)$.
  All together, 
  \begin{displaymath}
    \vol(K_X+D) = \vol(   f^*(K_X + D)) \ge \vol(K_Y+D'_Y) \ge v_1(\mc)
  \end{displaymath}
  Together with equation \eqref{eq:vols} this proves the statement.
\end{proof}

Next, we find a section  of a multiple of $K_X+B$ provided that
$\nklt(X,B)$ contains an isolated point.  

\begin{prop}\label{prop: nonvan}
  Let $(X,B)$ be in $\ms(\mc_2)$ such that $K_X+B$ is ample.
Let $p $ be an isolated nklt center of
  $(X,B)$. Assume that $m(K_X+B)$ is integral and Cartier near the
  point $p$. Then
\begin{enumerate}
\item If $m\geq 2$ then 
$\dim H^0(X, \lfloor (m(K_X+B)\rfloor )\geq 1$.
\item If $m=1$ then $\dim H^0(X, \lfloor (2(K_X+B)\rfloor )\geq 1$.
\end{enumerate}
\end{prop}
\begin{proof}
  Part (ii) is obtained from (i) by using $2(K_X+B)$, so we only need to prove (i).
  Let $f\colon Y\rightarrow X$ be a log resolution. Write
  $K_Y+B_Y=f^*(K_X+B)$ and $E:=B_{Y}^{=1}$. By assumption, the divisor
  $E$ is reduced, nonempty, and has a connected component contained
  in $f\inv(p)$.  Since $(m-1)(K_Y+B_Y)$ is big and nef, we
  have $$H^1(Y, K_Y+\lceil (m-1)(K_Y+B_Y\rceil))=0$$ by the
  Kawamata--Viehweg vanishing theorem. The standard exact sequence
  $0\to \cO_Y(-E)\to \cO_Y\to \cO_E\to 0$ together with the above
  vanishing gives
  \begin{displaymath}
    H^0(Y, K_Y+\lceil (m-1)(K_Y+B_Y)\rceil + E)
    \onto 
    H^0(E, (K_Y+\lceil (m-1)(K_Y+B_Y)\rceil+ E )\restr{E})
  \end{displaymath}
  Let $E_0$ be a connected component of $E$ contained in
  $f\inv(p)$. The coefficients of the divisors $B_j$ in $B_Y$ which
  intersect $E_0$ satisfy $0\le b_j<1$; they appear in the
  different. 
  Since $m(K_X+B)$ is integral and Cartier at $p$, it
  follows that in a neighborhood of $E_0$ one has
  \begin{displaymath}
    K_Y+\lceil (m-1)(K_Y+B_Y)\rceil+ E = m(K_Y+B_Y)
  \end{displaymath}
  and is trivial. Therefore, the above two $H^0$ groups are
  nonzero. Thus, the divisor 
  \[
    f_*(K_Y+\lceil (m-1)(K_Y+B_Y)\rceil + E)=
    K_X + \lceil (m-1)(K_X+B)\rceil + B^{=1}
  \]
  is effective. Now it is easy to see that for a number of the form
  $b=1-\frac1{n}$, $n\in\bN$ and for any $m\in\bN$ one has
  $\lceil (m-1)b\rceil = \lfloor mb\rfloor$. Thus, the last divisor equals
  \begin{displaymath}
    K_X + \lceil (m-1)(K_X+B)\rceil + B^{=1} =
    \lfloor m(K_X+B) \rfloor.
  \end{displaymath}
\end{proof}

\begin{rmk}\label{rmk: nonvan}
Let $(X,B)\in \ms(\mc_2)$ with ample $K_X+B$ and  $k$ distinct isolated nklt centers $p_i, 1\leq i\leq k$. Assume that $m(K_X+B)$ is integral and Cartier near each point $p_i$.  Then the same argument of Proposition~\ref{prop: nonvan} gives:
\begin{enumerate}
\item If $m\geq 2$ then 
$\dim H^0(X, \lfloor (m(K_X+B)\rfloor )\geq k$.
\item If $m=1$ then $\dim H^0(X, \lfloor (2(K_X+B)\rfloor )\geq k$.
\end{enumerate}
\end{rmk}

\begin{prop}\label{prop:fixed-multiple}
  For each DCC set $\mc\subset(0,1]\cap\bQ$, there exists 
  $m(\cC)\in\bN$ such that for any  log surface $(X,B)$ in $\ms(\mc)$ with ample
  $K_X+B$ and a point $p\in X$ that is an isolated nklt
  center of $(X,B)$, the divisor $m(\cC)(K_X+B)$ is integral and
  Cartier near $p$. 
\end{prop}
\begin{proof}
  Let $f\colon Y\to X$ be a dlt blowup, as in 
  \eqref{def:dlt-blowup}.
  Let $B_Y=E+\Delta$ be the decomposition into
  the integral and fractional parts. We have $E\ne 0$ by assumption;
  write $E=\sum E_i$. 
  By Shokurov's connectedness theorem
  \cite[Lemma 5.7]{shokurov1992three-dimensional}, $E$ is connected.
  As in Section~\ref{sec:preliminaries}, the adjunction to $E$ gives
  \begin{displaymath}
    0 \equiv (K_Y+ E + \Delta) |_E = K_E + \Diff_E(\Delta),
    \qquad \Diff_E(\Delta) = \sum b_k' Q_k, \ b_k'\in \cC'.
  \end{displaymath}
  It follows that
  \begin{enumerate}
  \item Either $p_a(E)=1$, $E$ is a smooth elliptic curve or a cycle
    of $\bP^1$s, and $p\notin B$. Then $K_X+B$ is Cartier near $p$.
  \item Or $E$ is a $\bP^1$ or a chain of $\bP^1$'s and $\Delta$
    intersects only the end curves. 
  \end{enumerate}

  Restricting to an end curve gives an identity $\sum b'_k =1$ or $2$,
  with $b'_k<1$, $b'_k\in \cC'$. Since $\cC'$ is a DCC set, this identity has
  only finitely many solutions. 
  For a fixed $b_k'<1$ we have $b_k'=1 - \frac{ 1- \sum n_j b_j}{n}$ and
  there are only finitely many solutions for $n, n_j, b_j$. By
  Lemma~\ref{lem:different}, $n$ is the index of the singularity $Q_k\in Y$.

  The end result is that there are only finitely many possibilities
  for $b_j$ and the Cartier indices of $K_Y+B_Y$ at the singular points over
  $p$. Since all $b_j\in\bQ$ by assumption, there exists a fixed
  multiple $m$ depending only on $\cC$ such that $m(K_Y+B_Y)$ is
  integral and Cartier. If $p$ is a point as in (i) above then near
  $p$ the divisor $K_X+B=K_X$ is already Cartier. Otherwise $X$ is klt
  at $p$. By the Base Point Free Theorem for lc surfaces it follows
  that $m(K_X+B)$ is Cartier.
\end{proof}

\begin{lem}\label{lem:m-C2}
  Let $(X,B)$ be a log surface in $\ms(\mc_2)$  and let $p\in X$ be a point
  that is an isolated nklt center of $(X,B)$. Then the divisor
  $m(K_X+B)$ is integral and Cartier near $p$ for $m=1,2,3,4$ or
  $6$.
\end{lem}
\begin{proof}
  As in the proof of the above proposition, for a dlt blowup
  $Y\to X$ on an exceptional divisor $E=\sum E_i$ there are several
  singularities of index $m_j$ and an identity of the form
  $\sum (1-\frac1{n_j})=1$ or $2$.  As it is well known, the only
  solutions $(n_j)$ to this identity are $(2,2)$, $(3,3,3)$,
  $(2,4,4)$, $(2,3,6)$, $(2,2,2,2)$. By \eqref{cor:diff-C2}, the
  indices $m_j$ divide $n_j$: $n_j=m_jn'_j$ and the respective
  coefficients of $B$ are $b_j=1-\frac1{n'_j}$.  If $m$ is the GCD of
  these numbers (i.e.  $m=1,2,3,4$ or $6$) then $m(K_Y+B_Y)$ is
  integral and Cartier. So is $m(K_X+B)$ by the Base Point Free
  Theorem.
\end{proof}

Putting this together we get:

\begin{thm}\label{thm:C2-lower-bound}
  For any log surface $(X,B)\in\ms(\mc_2)$ with ample
  $K_X+B$ and $\nklt(X,B)\ne~\emptyset$ one has
  \begin{displaymath}
    (K_X+B)^2 \ge \frac{1}{7^2\cdot 42^2}
  \end{displaymath}
\end{thm}
\begin{proof}
  If $\lfloor B\rfloor\ne0$ then $(K_X+B)^2\ge \frac1{42^2}$ by
  Theorem~\ref{thm:v1}. Otherwise, there is an isolated nklt
  center. Then we apply Theorem~\ref{thm: low bnd} and use that
  $t_m = 1$ by Lemma~\ref{lem:t_m} and $m\le6$ by
  Lemma~\ref{lem:m-C2}. 
\end{proof}

We conclude by completing the proof of Theorem~\ref{thm:bounds}.

\begin{proof}[Proof of Theorem~\ref{thm:bounds}]
  (ii) The lower bound is proved in
  Theorem~\ref{thm:C2-lower-bound}. The upper bound is obtained by
  applying Theorem~\ref{thm: main} to
  \cite[Example~5.3.1]{kollar1994log-surfaces}.

  (iii) The lower bound follows from that for $\cC_2$. The upper bound
  is obtained by applying Theorem~\ref{thm: main} to the example in
  \cite{alexeev2016open-surfaces} of a strictly log canonical surface
  $X$ with ample $K_X$ and $K_X^2=\frac1{462}$.

  (i) For $\cC_2$, this follows from Corollary~\ref{cor: =}. For $\cC_0$, and
  $\cC_1$, using \eqref{cor: rat}, we have to rule out the possibility
  that the minimum is achieved on a log surface $(X,B)$ with an
  \emph{inaccessible} nklt locus, i.e. one that has 
  \begin{enumerate}
  \item[(a)] either a simple elliptic singularity,
  \item[(b)] or a smooth component $B_j$ of $B$ lying in a smooth part.
  \end{enumerate}
  However, in the case (a) one has $(K_X+B)^2\geq \frac{1}{143}$ by
  \cite{Liu2017minimal-volume}. And in the case (b)
  $(K_X+B)B_0\in \bN$, so by \cite{kollar1994log-surfaces}
  \begin{displaymath}
    (K_X+B)^2\ge \delta_1(\cC_2)(K_X+B)B_0 \ge \delta_1(\cC_2),  
  \end{displaymath}  
  where $\delta_1 = \sup(t)$ such that $K_X+B-t B_0$ is big. One has
  $\delta_1(\cC_2)=\frac1{42}$ by
  \cite[Thm. 5.3]{kollar1994log-surfaces}.  The existing limit point
  $\frac1{462}$ is smaller.
\end{proof}

\begin{rmk}
\begin{enumerate}[leftmargin=*]
\item By Lemma~\ref{lem:m-C2}, we can take $m(\cC_2) = 12$. If
  $(X, B)\in \ms(\mc_2)$ has ample $K_X+B$ and at least two isolated
  nklt centers, then one has
  $\dim H^0(X, \lfloor 12(K_X+B)\rfloor)\geq 2$ by Remark~\ref{rmk:
    nonvan}. In this case, one expects a better lower bound on
  $\vol(K_X+B)$, cf.~\cite[Theorem 5.1]{Bla95b}.
\item For any $\epsilon >0$, the set of 
numbers in $\kk(\mc_2)$ that are less than $\frac{1}{86436}-\epsilon$
is finite.
\item Koll\'ar \cite[Section 2]{kollar1994log-surfaces} gives a log canonical surface
  $(X, B)\in \ms(\mc_2)$ such that $K_X+B$ is ample and
  $(K_X+B)^2=\frac{1}{42^2\cdot 43^2}=\frac{1}{3261636}$.
  On the other hand, the current record
  for the  smallest (positive) volume of 
  $(X, B)\in \ms(\mc_0)$ and $\ms(\mc_1)$
  is $\frac{1}{48983}$ (\cite{alexeev2016open-surfaces}).
\end{enumerate}
\end{rmk}

\section{Iterated accumulation points of volumes}\label{sec:iterated}
\begin{defn}
  Let $V\subset \RR$ be a non-empty subset. A point $v\in \RR$ is
  a \emph{$k$-iterated accumulation point} of $V$ if
  $v\in \acc^k (V)$, where $\acc^k (V)\subset \RR$ is defined
  inductively as the set of accumulation points of $\acc^{k-1}(V)$, and
  $\acc^0(V):=V$.
\end{defn}
Obviously, $\acc^1(V)=\acc(V)$ is exactly the set of accumulation
points of $V$; $\acc^{k}(V)\subset \acc^{k-1}(V)$ for $k\geq 2$.
\begin{thm}
  For any set $\cC\subset (0,1]$ and any $k\in\bN$, one has
  $\acc^k(\bK^2(\mc))\ne\emptyset$. 
\end{thm}
\begin{proof}
  Consider $(X,B)=(\PP^2,\sum_{j=1}^n L_j)$ with $n\geq 4$
where
  the $L_j$'s are $n$ distinct lines in general position. Then
  $(X, B)$ is a log canonical surface with ample $K_X+B$ and
  $(K_X+B)^2=(n-3)^2$. 
  Let $p_j\in L_j\cap L_n$ for $1\leq j\leq n-1$ be the $n-1$ nodes of
  $B$ on $L_n$. We apply the construction of Theorem~\ref{thm: nklt}
  to all of these points at once.

  Let $h\colon Z=Y^{(s_1,\dotsc, s_{n-1})} \to \bP^2$ be obtained by blowing
  up $s_j$ times at the point $p_j$ and its preimages on the the strict transforms of $L_j$,
  $1\le j\le n-1$. As in Theorem~\ref{thm: nklt}, we define the
  divisor $B_Z$ by $K_Z+B_Z = h^*(K_X+B)$ and the
  divisor $B_Z'$ by 

  \begin{displaymath}
    B_Z' 
    = B_Z - \sum_{j=1}^{n-1} \sum_{i=1}^{s_j} \frac{i}{s_j} E_{j,i}
    = h_*\inv B + \sum_{j=1}^{n-1} \sum_{i=1}^{s_j} \frac{s_j-i}{s_j} E_{j,i}
  \end{displaymath}

  By construction, $K_Z+B_Z'$ intersects the curves $E_{j,i}$
  non-negatively, and one has $(K_Z+B_Z') h_*\inv L_j = n-4$ for
  $1\le j\le n-1$ and $n-3-\sum\frac1{s_j}$ for $j=L_n$. Thus, for
  $s_j\gg0$ the divisor $K_Z + B_Z'$ is nef of volume
  \begin{displaymath}
    (K_Z+B_Z')^2 = (n-3)^2 - \sum_{j=1}^k \frac1{s_j}
  \end{displaymath}
  and on the log canonical model of $(Z,B_Z')$ the curve $h_*\inv L_n$
  is not contracted and $(Z,B_Z')$ has a nonempty accessible nklt locus.

  Each of the above numbers is an accumulation point of $\bK^2(\cC_0)$
  by Theorem~\ref{thm: main}.  Considering the sequences with
  $s_1\gg s_2\gg \dotsc s_{n-1}\gg0$, we see that
  $(n-3)^2$ is in $\acc^{n}(\cC_0)$.  Now, for any $n\ge\max(k,4)$ one has
  $\acc^k(\cC)\supset\acc^n(\cC_0)\ne\emptyset$.
\end{proof}
\bibliographystyle{amsalpha}

\begin{thebibliography}{Bau09}

\bibitem[AL16]{alexeev2016open-surfaces}
Valery Alexeev and Wenfei Liu, \emph{Open surfaces of small volume},
arXiv:1612.09116, to appear in Algebraic Geometry. 

\bibitem[Ale92]{alexeev1992log-canonical-surface}
Valery Alexeev, \emph{{L}og canonical surface singularities: arithmetical
  approach}, Flips and abundance for algebraic threefolds, Soci\'et\'e
  Math\'ematique de France, Paris, 1992, Papers from the Second Summer Seminar
  on Algebraic Geometry held at the University of Utah, Salt Lake City, Utah,
  August 1991, Ast\'erisque No. 211 (1992), pp.~47--58. \MR{MR1225842
  (94f:14013)}

\bibitem[Ale94]{alexeev1994boundedness-and-ksp-2}
\bysame, \emph{Boundedness and {$K\sp 2$} for log surfaces}, Internat. J. Math.
  \textbf{5} (1994), no.~6, 779--810. \MR{MR1298994 (95k:14048)}

\bibitem[AM04]{alexeev2004bounding-singular}
Valery Alexeev and Shigefumi Mori, \emph{Bounding singular surfaces of general
  type}, Algebra, arithmetic and geometry with applications ({W}est
  {L}afayette, {IN}, 2000), Springer, Berlin, 2004, pp.~143--174. \MR{2037085
  (2005f:14077)}

\bibitem[Art62]{artin1962some-numerical}
Michael Artin, \emph{{S}ome numerical criteria for contractability of curves on
  algebraic surfaces}, Amer. J. Math. \textbf{84} (1962), 485--496.
  \MR{MR0146182 (26 \#3704)}


\bibitem[Bla95a]{blache1995example-concerning}
Raimund Blache, \emph{An example concerning {A}lexeev's boundedness results on
  log surfaces}, Math. Proc. Cambridge Philos. Soc. \textbf{118} (1995), no.~1,
  65--69. \MR{1329459}
  
  \bibitem[Bla95b]{Bla95b}
  \bysame, \emph{Riemann--Roch theorem for normal surfaces and applications}, Abh. Math. Sem. Univ. Hamburg {\bf 65} (1995), 307--340.
  
\bibitem[Fuj11]{fujino2011fundamental-theorems}
Osamu Fujino, \emph{Fundamental theorems for the log minimal model program},
  Publ. Res. Inst. Math. Sci. \textbf{47} (2011), no.~3, 727--789. \MR{2832805} 

\bibitem[KM98]{kollar1998birational-geometry}
J{\'a}nos Koll{\'a}r and Shigefumi Mori, \emph{{B}irational geometry of
  algebraic varieties}, Cambridge Tracts in Mathematics, vol. 134, Cambridge
  University Press, Cambridge, 1998, With the collaboration of C. H. Clemens
  and A. Corti, Translated from the 1998 Japanese original. \MR{MR1658959
  (2000b:14018)}

\bibitem[Kol94]{kollar1994log-surfaces}
J\'anos Koll\'ar, \emph{Log surfaces of general type; some conjectures},
  Classification of algebraic varieties ({L}'{A}quila, 1992), Contemp. Math.,
  vol. 162, Amer. Math. Soc., Providence, RI, 1994, pp.~261--275. \MR{1272703}

\bibitem[Kol08]{kollar2008is-there}
\bysame, \emph{Is there a topological {B}ogomolov-{M}iyaoka-{Y}au inequality?},
  Pure Appl. Math. Q. \textbf{4} (2008), no.~2, Special Issue: In honor of
  Fedor Bogomolov. Part 1, 203--236. \MR{2400877}

\bibitem[Liu17]{Liu2017minimal-volume}
Wenfei Liu, \emph{The minimal volume of log surfaces of general type with
  positive geometric genus}, Preprint (2017), arXiv:1706.03716.

\bibitem[Sho92]{shokurov1992three-dimensional} 
V.~V. Shokurov, \emph{Three-dimensional log perestroikas}, Izv. Ross. Akad.
  Nauk Ser. Mat. \textbf{56} (1992), no.~1, 105--203. \MR{1162635}

\bibitem[UY17]{urzua2017notes-accumulation}
Giancarlo Urz\'ua and Jos\'e~Ignacio Y{\'a}{\~n}ez, \emph{Notes on accumulation
  points of {$K^2$}}, Preprint (2017).

\end{thebibliography}

\def\cprime{$'$}
\providecommand{\bysame}{\leavevmode\hbox to3em{\hrulefill}\thinspace}
\providecommand{\MR}{\relax\ifhmode\unskip\space\fi MR }
\providecommand{\MRhref}[2]{%
  \href{http://www.ams.org/mathscinet-getitem?mr=#1}{#2}
}
\providecommand{\href}[2]{#2}

\end{document}